\documentclass[12pt,oneside]{amsart}
      \usepackage{amssymb}
      \usepackage{enumerate}
\usepackage{setspace}
\usepackage[margin=1in]{geometry}  
\usepackage{ulem}       
\usepackage{amsmath, amsthm, amssymb}
  \usepackage{verbatim} 
\usepackage{xspace}
\usepackage{mathrsfs}
\usepackage[pdftex,bookmarks,colorlinks,breaklinks]{hyperref}  
\hypersetup{linkcolor=blue,citecolor=red,filecolor=dullmagenta,urlcolor=darkblue} 
 \usepackage{graphicx}  
 \usepackage{color}	
 \usepackage{mathabx}
\newtheorem{theorem}{Theorem}[section]
\newtheorem{corollary}[theorem]{Corollary}
\newtheorem{lemma}[theorem]{Lemma}
\newtheorem{conjecture}[theorem]{Conjecture}

\newtheorem{remark}[theorem]{Remark}
\newtheorem{example}[theorem]{Example}
\newtheorem{question}{Problem}

\newcommand{\cc}{\mathcal}
\newcommand\crk{\operatorname{cr}}

      \makeatletter
      \def\@setcopyright{}
      \def\serieslogo@{}
      \makeatother
   
\begin{document}
\thispagestyle{empty}

   \author{Amin  Bahmanian}
   \address{Department of Mathematics,
  Illinois State University, Normal, IL USA 61790-4520}

   \author{Sadegheh Haghshenas}
   \address{Department of Mathematics,
  Illinois State University, Normal, IL USA 61790-4520}
  \email{shaghsh@ilstu.edu}

\title[Partitioning The Edge Set of a Hypergraph Into Almost Regular Cycles]{Partitioning The Edge Set of a Hypergraph Into Almost Regular Cycles}

\begin{abstract}
A  cycle of length $t$ in a hypergraph is an alternating sequence $v_1,e_1,v_2\dots,v_t,e_t$ of distinct vertices $v_i$ and distinct edges $e_i$  so that  $\{v_i,v_{i+1}\}\subseteq e_i$ (with $v_{t+1}:=v_1$). Let $\lambda K_n^h$ be the  $\lambda$-fold $n$-vertex complete $h$-graph. Let $\cc G=(V,E)$ be a hypergraph all of whose edges are of size at least $h$, and  $2\leq c_1\leq \dots\leq c_k\leq |V|$.
In order to partition the edge set of  $\mathcal G$ into cycles of specified lengths $c_1, \dots, c_k$, an obvious necessary condition is that $\sum_{i=1}^k c_i=|E|$. We show that this condition is sufficient in the following cases. 
\begin{enumerate}
\item [(R1)] $h\geq \max\{c_k, \lceil n/2 \rceil+1\}$.
\item [(R2)] $\mathcal G=\lambda K_n^h$, $h\geq \lceil n/2 \rceil+2$.
\item [(R3)] $\mathcal G=K_n^h$, $c_1= \dots=c_k:=c$, $c|n(n-1), n\geq 85$.
\end{enumerate}
In (R2), we  guarantee that each cycle is almost regular. In (R3), we also solve the case where a ``small" subset $L$ of edges of $K_n^h$ is removed.
\end{abstract}
   \date{\today}
   \maketitle

\section{Introduction}
In order to decompose (i.e. partition the edge set of) a graph $G=(V,E)$ into $k$ cycles of specified lengths $c_1,\dots,c_k$ with $3\leq c_1\leq \dots\leq c_k\leq |V|$, it is clearly necessary that $\sum_{i=1}^k c_i=|E|$, and that the degree of each vertex is even. In a recent breakthrough by Bryant et. al. \cite[Theorem 1]{MR3214677}, it is shown that these conditions are indeed sufficient for the case when $G$ is the  complete $n$-vertex graph, $K_n$, which settles an old conjecture of Alspach.

The {\it corank} of a hypergraph $\mathcal G=(V,E)$, written, $\crk(\mathcal G)$, is $\min \{|e|: e\in E\}$. The {\it degree} of a vertex is the number of edges containing it, and $\mathcal G$ is {\it regular} 
if all its vertices have the same degree. 
We call a hypergraph {\it almost regular} if there is an integer $d$ such that every vertex has degree $d$ or $d+1$. 
In this paper, our goal is to investigate cycle decompositions of  hypergraphs, and in particular we are interested to investigate the following problem.
\begin{question} \label{prob1cyc}
 Find the necessary and sufficient conditions for the existence of a decomposition of the complete $\lambda$-fold $n$-vertex $h$-graph,  $\lambda K_n^h$,  into $k$ regular cycles of specified lengths $c_1,\dots,c_k$. 
\end{question}
This problem is not precisely formulated, because it is not clear what the definition of `cycle' in a hypergraph should be. Let $\mathcal G=(V,E)$ be a hypergraph (possibly with multiple edges)  with $\crk(\mathcal G):=h$. A careful reader will notice that all of the following are potential definitions for $\mathcal G$ to be a cycle.
\begin{enumerate}[(I)]
\item There exists an alternating sequence $v_1,e_1,v_2,\dots,v_t,e_t$ of distinct vertices $v_i$, and distinct edges $e_i$ with $V\supseteq \{v_1,\dots,v_t\}, E= \{e_1,\dots,e_t\}$, and
 $\{v_i,v_{i+1}\}\subseteq e_i$ for $1\leq i\leq t$. (Here $v_{t+1} := v_1$ and $t\geq2$.)
\item  There exists a cyclic ordering of the vertices of $\mathcal G$ such that every edge of $\mathcal G$ consists of  consecutive vertices and such that every pair of consecutive edges (in the natural ordering of the edges) intersect.
\item  There exists an integer $\ell$ and  a cyclic ordering of the vertices of $\mathcal G$ such that every edge of $\mathcal G$ consists of  consecutive vertices and such that every pair of consecutive edges (in the natural ordering of the edges) intersect in precisely $\ell$ vertices.  
\item  $\mathcal G$ is  2-regular and connected. 
\item  $\mathcal G$ is $h$-regular and connected.
\end{enumerate}
A   common property of $\mathcal G$ in all these definitions is the connectivity. For the purpose of this paper, $\mathcal G$ will be called a  {\it cycle} (Berge cycle),  {\it circle}, {\it $\ell$-intersecting circle}, {\it connected 2-factor}, and {\it connected $h$-factor}, respectively, if it satisfies condition (I), (II), (III), (IV), and (V), respectively. Note that every circle is also a cycle, but the converse is not  true. 

By the {\it length} of a cycle we mean the number of edges of the cycle, and a {\it $t$-cycle} is a cycle of length $t$. 
In an $n$-vertex hypergraph, a cycle is {\it Hamiltonian} if it is of length $n$, and {\it spanning} if it spans the whole hypergraph. These two notions coincide when $\cc G$ is a graph. Observe that every Hamiltonian cycle is spanning, but the converse is not true.

For example, $\cc G_1=(\{1,2,3,4,5,6\},\{ 132, 264, 435, 561\})$ is a 2-regular non-Hamiltonian 4-cycle (and a connected 2-factor), but it is not a circle, and  $\cc G_2=(\{1,2,3,4,5,6\},\{312, 264, 435, 531\})$ is an irregular circle, but there is no $\ell$ for which $\cc G_2$ is $\ell$-intersecting. The cycles $\mathcal C$ with edge set $\{132, 253, 314, 415, 531\}$ and $\mathcal C'$ with edge set $\{243, 354, 421, 125, 542\}$ provide  a  (cyclic) decomposition of $K_5^3$ into irregular Hamiltonian cycles. Note that  an $h$-uniform $(h-1)$-intersecting circle is $h$-regular. In the last section, we will show that not every connected factor is a cycle. 

A very special case of Problem \ref{prob1cyc} where $\lambda=1$, and each  (not necessarily regular) cycle is Hamiltonian  was in fact studied by Bermond  in the 70s \cite{MR0505807}. In order to decompose $K_n^h$ into Hamiltonian cycles, it is clearly necessary that $n|\binom{n}{h}$. The results of Bermond \cite{MR0505807}, Verrall \cite{MR1297389}, K\"uhn and Osthus \cite{MR3213309} collectively imply that this condition is  sufficient as long as $n$ is not too small.  

Working with circles seems to be much more challenging. When $\lambda=1$ and each circle is  Hamiltonian and $(h-1)$-intersecting,  Bailey and Stevens  \cite{MR2684077} solved  Problem \ref{prob1cyc} for the cases  when  (i) $h=3, n\leq 16$, and (ii) $(n,h)=(9,4)$. Later,  Meszka and Rosa \cite{MR2554542} solved the cases when (i) $\lambda=1,h=3, n\leq 32$,  (ii) $\lambda=1, (n,h)=(13,4)$, (iii) $\lambda=h=3, n$ even. Using design theoretic ideas, they also solved the problem of decomposing $K_n^3$ into $2$-intersecting 4-cycles. 
 
In a different direction the following related problem is solved  by the first author \cite{MR3213845}: $\lambda K_n^h$ can be decomposed into connected spanning regular subgraphs $G_1, \ldots, G_k$ such that each $G_i$ has $c_i$ edges if and only if (i) $\sum_{i=1}^k c_i=\lambda \binom{n}{h}$, (ii) $c_ih/n\geq 2$ and (iii) $c_ih/n$ is an integer, for $1\leq i\leq k$.

In order to decompose  $\mathcal G=(V,E)$ with $\crk(\mathcal{G})=h$ into cycles of specified lengths $c_1, \dots, c_k$ with $2\leq c_1\leq  \dots\leq  c_k\leq |V|$, an obvious necessary condition is that $\sum_{i=1}^k c_i=|E|$. We show that this condition is sufficient in the following cases. 
\begin{enumerate}
\item [(R1)] $h\geq \max\{c_k, \lceil n/2 \rceil+1\}$.
\item [(R2)] $\mathcal G=\lambda K_n^h$, $h\geq \lceil n/2 \rceil+1$.
\item [(R3)] $\mathcal G=K_n^h$, $c_1= \dots=c_k:=c$, $c|n(n-1), n\geq 85$.
\end{enumerate}
In (R2), we  guarantee that each cycle is almost regular. In (R3), we also solve the case where a  subset $L$ of edges of $K_n^h$ with $|L|<c$ is removed.

We shall prove the cases (R1) and (R2) in Section \ref{hicr}, and the case (R3) in Section \ref{fixedlen}. The results of Section \ref{hicr} rely on finding long cycles in bipartite graphs. To ensure that the cycles in a decomposition are (almost) regular, an extension of Baranyai's Theorem is used. The main result of Section \ref{fixedlen} heavily relies on that of K\"uhn and Osthus \cite{MR3213309} which employs the Kruskal-Katona theorem. 

We end this section with some notation. For a subset $S$ of vertices of a graph $G$, let $N(S)$ be the set of vertices having a neighbor in $S$, and  $[n]:=\{1,2,\dots,n\}$. A bipartite graph $G$ with bipartition $\{X, Y \}$ will be denoted by $G[X, Y ]$.

\section{Hypergraphs of High Corank} \label{hicr}
In this section, we focus on $n$-vertex hypergraphs of corank at least $\lceil n/2 \rceil+1$. 

To prove our first result, we need the following theorem.
\begin{theorem}  \textup{(Jackson, \cite[Theorem 1]{MR624549})} \label{Jack78bip}
Let $h\geq 2$, and $G[X,Y]$ be a simple bipartite graph  with $2\leq |X|\leq h\leq  |Y|\leq 2h-2$, such that the degree of each vertex in $X$ is at least $h$. Then $G$ contains a cycle of length $2|X|$. 
\end{theorem}
Now, we are ready to prove the following result.
\begin{theorem}\label{largeh}
Let $\cc G=(V,E)$ be a hypergraph, and $2\leq c_1\leq \ldots\leq c_k\leq|V|$ such that $\crk(\cc G)\geq \max\{c_k, \lceil \frac{|V|}{2} \rceil+1\}$. Then  $\cc G$  can be decomposed into $k$  cycles of specified  lengths $c_1,\dots,c_k$ if and only if   $\sum_{i=1}^k c_i=|E|$.  
\end{theorem}
\begin{proof} The necessity is obvious. To prove the sufficiency, note that since  $\sum_{i=1}^k c_i=|E|$, we can decompose $\mathcal G$ into subgraphs $F_1,\dots,F_k$ with $|E(F_i)|=c_i$ for $i\in[k]$.  Let $h=\crk(\mathcal G)$, and  fix an element $F\in \{F_1,\dots,F_k\}$. Define the bipartite simple graph $B[X,Y]$ with $X:=E(F),Y:=V(F)$ in which two vertices $x,y$ with $x\in X,y\in Y$ are adjacent, if and only if $y\in x$ in $F$. Since  $2\leq c_i\leq h$ for $i\in[k]$, we have that $2\leq |X|\leq h$. Moreover, since $\crk(\mathcal G)=h$, the degree of each vertex in $X$ is at least $h$, and since $h\geq \lceil \frac{|V|}{2} \rceil+1$, we have $h\leq  |Y|\leq 2h-2$. Therefore, by Theorem \ref{Jack78bip} $B$ contains a cycle $C$ of length $2|X|$.
Let us assume that $C$ is the sequence of distinct vertices  $v_1,e_1,\dots,v_{|X|},e_{|X|},v_1$ with $v_i\in Y,e_i\in X$ for $1\leq i\leq|X|$. 
Clearly, $v_1,e_1,\dots,v_{|X|},e_{|X|}$ corresponds to a cycle of length $|X|$ in $\mathcal G$, and therefore $F$ is a cycle of length $|X|$ in $\mathcal G$.  
Therefore, $\cc G$ can be decomposed into $F_1,\dots,F_k$ where $F_i$ is a cycle of length $c_i$ for $i\in[k]$. 
\end{proof}
To prove our next result,  we need a few preliminaries.  

Let $G[X,Y]$ be a simple bipartite graph with $|X|=|Y|\geq 2$. By  a  classical result of Moon and Moser \cite[Corollary 3]{MR0161332}, if $d(x)+d(y)\geq |X|+1$ for every pair of non-adjacent vertices $x\in X,y\in Y$, then $G$ is Hamiltonian. For the purpose of this paper, we need the following   generalization.
\begin{theorem} \textup{(Chiba et al., \cite[Theorem 5]{MR2913078})}  \label{moonmosergen}
Let $G[X,Y]$  be a  bipartite graph with $ |X|\geq  |Y|\geq 2$. If $d(x)+d(y)\geq \frac{1}{2}(|X|+|Y|)+1$ for every pair of non-adjacent vertices $x\in X,y\in Y$, then $G$ contains a cycle of length $2|Y|$.
\end{theorem}

Recall that a hypergraph is said to be almost regular if the degree of all its vertices are within one of another. A complete solution to the problem of factorization of $K_n^h$ was  obtained   in the 1970s by Baranyai \cite{MR0416986}. 
For the purpose of this paper the following $\lambda$-fold version of Baranyai's Theorem is needed. 
\begin{theorem}   \textup{(Bahmanian, \cite[Theorem 6.4]{MR2942724})}  \label{lambdabaranyai}
$\lambda K_n^h$ can be decomposed into spanning subgraphs $F_1,\dots,F_k$ with $|E(F_i)|=c_i$ so that each $F_i$ is almost regular if and only if  $\sum_{i=1}^k c_i=\lambda \binom{n}{h}$. 
\end{theorem}
Now, we are ready to prove our next result.
\begin{theorem} \label{thm1hchyp}
Let $c_1,\dots,c_k,n,h,\lambda \in \mathbb{N}$  with $2\leq c_1\leq \dots\leq c_k\leq n$, and $h\geq \frac{n}{2}+\frac{n(1+\epsilon_i)}{n+c_i}$ for $i\in[k]$ where $\epsilon_i=c_ih/n-\lfloor c_ih/n\rfloor$. Then $\lambda K_n^h$ can be decomposed into $k$ almost regular  spanning cycles of specified  lengths $c_1,\dots,c_k$ if and only if  $\sum_{i=1}^k c_i=\lambda \binom{n}{h}$. Moreover, for each $i$ with $\epsilon_i=0$, the corresponding   cycle of length $c_i$ is regular.
\end{theorem} 

\begin{proof} 
In order to decompose $\lambda K_n^h$ into $k$ (almost regular spanning) cycles of lengths $c_1,\dots, c_k$, it is clearly necessary that $\sum_{i=1}^k c_i=\lambda \binom{n}{h}$. 
To prove the sufficiency, note that since $\sum_{i=1}^k c_i=\lambda \binom{n}{h}$, by Theorem \ref{lambdabaranyai} $\lambda K_n^h$ can be decomposed into spanning subgraphs $F_1,\dots,F_k$ with $|E(F_i)|=c_i$ so that each $F_i$ is almost regular. Fix an $i$, $i\in[k]$, and define the bipartite simple graph $B_i[X_i,Y_i]$ with $X_i:=V(F_i),Y_i:=E(F_i)$ in which two vertices $x,y$ with $x\in X_i,y\in Y_i$ are adjacent, if and only if  $x\in y$ in $F_i$. Since $F_i$ is $h$-uniform and almost regular,  in the bipartite graph $B_i$, the degree of each vertex in $Y_i$ is $h$, and the degree of each vertex in $X_i$ is either $\lfloor c_ih/n \rfloor$ or $\lceil c_ih/n\rceil$. This implies that the minimum degree of any two non-adjacent vertices from  different parts of $B_i$ is at least $h+\lfloor c_ih/n \rfloor$. Now, we have 
\begin{eqnarray*}
h \geq \frac{n}{2}+\frac{n(1+\epsilon_i)}{n+c_i}=\frac{n(n+c_i+2+2\epsilon_i)}{2(n+c_i)}=\frac{n+c_i+2+2\epsilon_i}{2(1+c_i/n)} &\iff & \\
h(1+\frac{c_i}{n})\geq \frac{c_i+n}{2}+1+\epsilon_i & \iff & \\
h+\frac{c_i h}{n}-\epsilon_i \geq \frac{c_i+n}{2}+1& \iff & \\
h+\left\lfloor \frac{c_i h}{n} \right\rfloor\geq \frac{c_i+n}{2}+1.
\end{eqnarray*}
Therefore, by Theorem \ref{moonmosergen},  $B_i$ contains a cycle of length $2c_i$. 
Let us assume that $C$ is the sequence of distinct vertices  $v_1,e_1,\dots,v_{c_i},e_{c_i},v_1$ with $v_j\in X_i,e_j\in Y_i$ for $1\leq j\leq c_i$. 
Clearly, $v_1,e_1,\dots,v_{c_i},e_{c_i}$ corresponds to a spanning cycle of length $c_i$ in $\mathcal G$, and therefore $F_i$ is an almost regular spanning cycle of length $c_i$ in $\mathcal G$.  
Therefore, $\cc G$ can be decomposed into $F_1,\dots,F_k$ where $F_i$ is an almost regular spanning cycle of length $c_i$ for $i\in[k]$. 

If for some $i$, $\epsilon_i=0$, then $F_i$ is $\frac{c_ih}{n}$-regular, and this completes the proof. 
\end{proof}
Note that for $i\in[k]$, $2\leq c_i\leq n$ and $0\leq \epsilon_i\leq \frac{n-1}{n}$, therefore we have $1/2\leq \frac{n(1+\epsilon_i)}{n+c_i}<2$.
\begin{corollary} \label{thm1hchypcor}
Let $c_1,\dots,c_k,n,h,\lambda \in \mathbb{N}$  with $2\leq c_1\leq \dots\leq c_k\leq n$, and $h\geq  n/2 +2$. Then $\lambda K_n^h$ can be decomposed into $k$ almost regular  spanning cycles of specified  lengths $c_1,\dots,c_k$ if and only if   $\sum_{i=1}^k c_i=\lambda \binom{n}{h}$. 
\end{corollary} 
\begin{corollary} \label{thm1hchypcor2}
If $h\geq  (n+1)/2$, then $\lambda K_n^h$ can be decomposed into  regular  Hamiltonian cycles  if and only if   $n | \lambda \binom{n}{h}$. 
\end{corollary} 

\begin{example}\textup{
 Since $\binom{100}{51}$ is divisible by 100, $K_{100}^{51}$ can be decomposed into regular   Hamiltonian cycles.
}\end{example}
\begin{example}\textup{
For any $h\in\{52,53,\dots, 99\}$, $K_{100}^{h}$ can be decomposed into almost regular spanning cycles of length $c_1,\dots,c_k$  if and only if $\sum_{i=1}^k c_i= \binom{100}{h}$.
}\end{example}
\begin{example}\textup{
For any non-negative integer solution $(a,b,c,d,e)$ of the equation  $5a+10b+15c+75d+100e=\binom{100}{60}$, there exists a decomposition of $K_{100}^{60}$  into  $a$ regular cycles of length 5,  $b$ regular cycles of length 10,  $c$ regular cycles of length 15,  $d$ regular cycles of length 75, and $e$   regular Hamiltonian cycles.
}\end{example}

\begin{remark} \label{singleedgeleaf} \textup{
In Theorem \ref{thm1hchyp} and Corollary \ref{thm1hchypcor}, we can replace the condition  $c_1\geq 2$  by $c_1\geq 1$, however with this new condition, those subgraphs (in the decomposition of $\lambda K_n^h$) with only one edge will not be cycles anymore. 
}\end{remark}

\section{Fixed Cycle Lengths} \label{fixedlen}
In this section, we prove the following result.
\begin{theorem}\label{ccycledecomposition}
 Let $n,h,c\in \mathbb{N}$ with $2\leq c\leq n, 3\leq h< n$, and let $\ell=\binom{n}{h}-c\lfloor \binom{n}{h}/c \rfloor$. 
 Then for any  $L\subseteq E(K_n^h)$ with $|L|=\ell$,   $K_n^h\backslash L$ can be decomposed into  cycles of length $c$ in the following cases:

\textup{(A)} $h=3$, $c|n(n-1)$, $n\geq85$, and $L$ is a matching.

\textup{(B)} $4\leq h< \max\left\{c,\left\lceil\frac{n}{2}\right\rceil+1\right\}$, $c|n(n-1)$, and $n\geq23$.

\textup{(C)} $h\geq \max\left\{c,\left\lceil\frac{n}{2}\right\rceil+1\right\}$.

\textup{(D)} $h=n-2$, $c\in\{n-1,n\}$, and $n\geq16$.

\textup{(E)} $h=n-1$ and $c=n$.   
\end{theorem} 

The proof heavily relies on that of \cite{MR3213309}. Using the following lemma, we transform the problem of decomposing $K_n^h$ into cycles, into a problem of decomposing a graph into cycles. 
\begin{lemma}\label{bipartitegperfectm}
Given an $h$-uniform hypergraph $\cc G$ and a graph $H$ with $V:=V(\cc G)=V(H)$, let $B[X,Y]$ be a bipartite graph with $X:=E(\cc G)$, $Y:=E(H)$, such that for $x\in X$, and $y\in Y$, $xy$ is an edge in $B$ if $y\subseteq x$. If $B$ has a perfect matching, and $H$ can be decomposed into cycles of lengths $c_1,\dots,c_k$, then $\cc G$ can be decomposed into cycles of lengths $c_1,\dots,c_k$.
\end{lemma}
\begin{proof}
Note that the vertices in $X$, and $Y$, correspond to $h$-subsets, and 2-subsets of $V$, respectively. Let $M$ be a perfect matching in $B$. 
Let $c\in\{c_1,c_2,\dots,c_k\}$ and $C$ be a $c$-cycle in the graph $H$. We can write $C$ as the sequence $v_1,f_1,v_2,f_2,\dots,v_c,f_c,v_1$, where for $i\in[c]$, $v_i\in V$ and $f_i\in Y$. 
For each $i\in[c]$, let $e_i$ be the vertex in $X$ such that $e_if_i\in M$.
By the definition of $B$, the sequence $v_1,e_1,v_2,e_2,\dots,v_c,e_c,v_1$ forms a $c$-cycle in the hypergraph $\cc G$. 
Thus, each edge in the given decomposition of $H$ corresponds to a cycle of the same length in $\cc G$. Since $M$ is a perfect matching and the decomposition of $H$ into cycles covers all edges of $H$ exactly once, we obtain the desired decomposition of $G$ into cycles of lengths $c_1,\dots,c_k$.
\end{proof}
We need the following result, which is an immediate consequence of a more general theorem on decompositions of complete digraphs into directed cycles.

\begin{theorem}  \textup{(Alspach et al. \cite[Theorem 1.1]{MR1986837})} \label{dknalspach}
For $c,n\in \mathbb{N}$ with $2\leq c\leq n$ and $(n,c)\notin\{(4,4),(6,3),(6,6)\}$, the graph $2K_n$ can be decomposed into cycles of length $c$ if and only if $c|n(n-1)$.
\end{theorem} 
Let $n,h$ be integers with $0\leq h\leq n$, and ${[n]\choose h}=\{s\subseteq[n]:|s|=h\}$. For $i\in[h]$ and for any $S$ with $S\subseteq{[n]\choose h}$, we define $\partial_i^-(S)$ and $\partial_i^+(S)$ as follows:
\begin{eqnarray*}
&&\partial_i^-(S)=\left\{t\in{[n]\choose h-i}:t\subseteq s \mbox{ for some } s\in S\right\},\\
&&\partial_i^+(S)=\left\{t\in{[n]\choose h+i}:s\subseteq t \mbox{ for some }  s\in S\right\}.
\end{eqnarray*}

For any positive real number $s$ and any integer $h$ with $1\leq h\leq s$, 
$${s\choose h}:=s(s-1)\dots(s-h+1)/h!.$$
We need the following two lemmas, which are consequences of the Kruskal-Katona theorem. 
\begin{lemma} \textup{(K\"uhn and Osthus \cite{MR3213309})}\label{KruKatCor}
Let $n,h\in\mathbb{N}$, $2\leq h\leq n$, and $\emptyset\neq T\subseteq{[n]\choose h}$.
\begin{enumerate}[]
\item \textup{(i)} 
If $h\geq 3$ and $t\in\mathbb{R}$ such that $\left|T\right|={t\choose h}$, then, $\left|\partial_{h-2}^-(T)\right|\geq{t\choose2}$.
\item \textup{(ii)} If $h=2$ and $t:=\left|T\right|\leq n-1$, then $$\left|\partial_2^+(T)\right|\geq t{n-t-1\choose2}+{t\choose2}(n-t-1).$$
\item \textup{(iii)} If $h=2$, $p,q\in\mathbb{N}\cup\{0\}$, such that $p<n$, $q<n-(p+1)$, and $\left|T\right|=pn-{p+1\choose2}+q$, then, 
\begin{eqnarray*}\left|\partial_1^+(T)\right|&\geq& {n-1\choose2}+{n-2\choose2}+\dots+{n-p\choose2}+q(n-p-2)-{q\choose2}\\
&\geq& p{n-p\choose2}+q(n-p-2)-{q\choose2}.
\end{eqnarray*}
\end{enumerate}
\end{lemma}
\begin{remark}\textup{
Note that in Lemma \ref{KruKatCor} (iii), $|T|$ is positive since 
\begin{eqnarray*}
|T|&=&pn-\dfrac{(p+1)p}{2}+q\geq p\left(\frac{2n-p-1}{2}\right)\\
&>&p\left(\frac{2n-n-1}{2}\right)=p\left(\frac{n-1}{2}\right)\geq0.
\end{eqnarray*}
}
\end{remark}
In particular, when $n\geq85$ and $p\leq8$, the bound in Lemma \ref{KruKatCor} (iii) can be simplified.
\begin{lemma}\label{KruKatCor+}
If $p$ and $q$ are defined as in Lemma \ref{KruKatCor} (iii), $n\geq 85$, and $p\leq8$, then $\left|\partial_1^+(T)\right|\geq p{n-p\choose2}+\frac{2qn}{5}$.
\end{lemma}
\begin{proof}
By Lemma \ref{KruKatCor} (iii), 
$$\left|\partial_1^+(T)\right|\geq p{n-p\choose2}+q(n-p-2)-{q\choose2}.$$
Note that $q\leq n-2$, since $q<n-(p+1)$ and $p\geq0$. Since $n\geq85$, $p\leq8$, and $q\leq n-2$, we have $10n-10p-20-5q+5\geq4n$, or equivalently, $n-p-2-\frac{q-1}{2}\geq\frac{2n}{5}$. Multiplying both sides by $q$, we have $q(n-p-2)-{q\choose2}\geq\frac{2qn}{5}$. This completes the proof.
\end{proof}
For the rest of this section, let $n$, $h$, $c$, and $L$ satisfy the conditions of Theorem \ref{ccycledecomposition}. The following parameters will be frequently used.
$$
\alpha:=\left\lfloor\frac{{n\choose h}-\left|L\right|}{n(n-1)}\right\rfloor,
\beta:=\frac{{n\choose h}-\left|L\right|-\alpha n(n-1)}{c}.
$$
By our assumption in Theorem \ref{ccycledecomposition}, $c|{n\choose h}-|L|$ and $c|n(n-1)$, and so $\beta\in\mathbb{N}\cup\{0\}$. Moreover, $\alpha>\frac{{n\choose h}-|L|}{n(n-1)}-1$ implies $\alpha n(n-1)>{n\choose h}-|L|-n(n-1)$, and so, $\beta c<n(n-1)$ or equivalently, $\beta\leq\frac{n(n-1)}{c}-1$. Therefore,
\begin{equation}\label{sc}
\beta c\leq n(n-1)-c.
\end{equation}
For graphs $G$ and $H$, let the union of $G$ and $H$, denoted by $G\cup H$, be the graph with vertex set $V(G)\cup V(H)$ and edge set $E(G)\cup E(H)$. Note that $G\cup H$ can be a multigraph. We also define the disjoint union of $G$ and $H$, denoted by $G+H$, to be the graph with vertex set $V(G)\cup V(H)$, where $V(G)$ and $V(H)$ are considered as disjoint sets, and edge set $E(G)\cup E(H)$. The disjoint union of graphs $G_1,G_2,\dots,G_k$ is denoted by $\sum_{i=1}^kG_i$.

Since $c|n(n-1)$, $2K_n$ can be decomposed into cycles of length $c$, by Theorem \ref{dknalspach}. We need the following definitions throughout this section:
\begin{itemize}
\item $H_1:=\sum_{i=1}^{\beta}C_i$, where $C_1,\dots,C_{\beta}$ are $\beta$ arbitrary cycles from a decomposition of $2K_n$ into $c$-cycles. Note that by (\ref{sc}), $H_1$ is well-defined. 
\item $H_2:=\sum_{i=1}^{\alpha}\left(B_{2i-1}\cup B_{2i}\right)$, where for $i\in[2\alpha]$, $B_i$ is a copy of $K_n$.
\item $\cc G:=K_n^h\backslash L$ and $H:=H_1+H_2$. Note that by the definition of $H_2$, $H$ is a multigraph, and that $\left|E(\cc G)\right|=\left|E(H)\right|$.
\item $B[X,Y]$ is the balanced bipartite graph with $X:=E(\cc G)$, $Y:=E(H)$, such that for $x\in X$, and $y\in Y$, $xy$ is an edge in $B$ if $y\subseteq x$.
\end{itemize}  
Since for $i\in[\beta]$, $B_{2i-1}\cup B_{2i}\cong2K_n$, by Theorem \ref{dknalspach}, $H_2$ and consequently $H$, can be decomposed into $c$-cycles. Therefore, by Lemma \ref{bipartitegperfectm}, in order to decompose $\cc G$ into $c$-cycles, it suffices to show that $B$ has a perfect matching. To do so, we will heavily rely on Lemmas \ref{KruKatCor} and \ref{KruKatCor+} to ensure that $B$ satisfies Hall's condition. To that end, let $S$ be an arbitrary nonempty subset of $X$. We will show that $|N(S)|\geq|S|$.

For the rest of this section, we will frequently use the following parameters. 
\begin{itemize}
\item $a:=\frac{|S|}{{n\choose h}}$. It is clear that $0\leq a\leq1$.
\item Let $s\in \mathbb{R}$, with $h\leq s\leq n$, such that $|S|={s\choose h}$. 
\item $b:=\frac{\left|N\left(S\right)\cap E(B_1)\right|}{{n\choose2}}$. It is clear that $0\leq b\leq1$.
\item $g:=\frac{{n\choose h}-\left|X\right|+n(n-1)-c}{{n\choose h}}$. Note that $1-g=\frac{|X|-n(n-1)+c}{{n\choose h}}$.
\end{itemize}
\begin{lemma}\label{preg}
$|N(S)|\geq a^{2/h}(|X|-n(n-1)+c)$.
\end{lemma}
\begin{proof}
Note that each element of $S$ is an $h$-subset of $[n]$ and $N(S)\cap E(B_1)=\partial_{h-2}^-(S)$. By Lemma \ref{KruKatCor} (i), $\left|N\left(S\right)\cap E(B_1)\right|\geq{s\choose2}$. Therefore, $b{n\choose2}\geq{s\choose2}$. Hence:
\begin{eqnarray*}
b^h&\geq&\frac{{s\choose2}^h}{{n\choose2}^h}=\frac{\left(s\left(s-1\right)\right)^h}{\left(n(n-1)\right)^h}=\left(\prod_{i=1}^h\frac{s}{n}\right)\left(\prod_{i=1}^h\frac{s-1}{n-1}\right)\\
&=&\left(\frac{s-1}{n-1}\prod_{i=1}^{h-1}\frac{s}{n}\right)\left(\frac{s}{n}\prod_{i=2}^h\frac{s-1}{n-1}\right)\\
&\geq&\left(\frac{s-h+1}{n-h+1}\prod_{i=1}^{h-1}\frac{s-i+1}{n-i+1}\right)\left(\frac{s}{n}\prod_{i=2}^h\frac{s-i+1}{n-i+1}\right)\\
&=&\left(\prod_{i=1}^h\frac{s-i+1}{n-i+1}\right)\left(\prod_{i=1}^h\frac{s-i+1}{n-i+1}\right)\\
&=&\left(\frac{s\left(s-1\right)\dots(s-h+1)}{n(n-1)\dots(n-h+1)}\right)^2=\frac{{s\choose h}^2}{{n\choose h}^2}=a^2.
\end{eqnarray*}
In the inequalities above, we have used the fact that $\frac{s}{n}\geq\frac{s-1}{n-1}\geq\frac{s-i+1}{n-i+1}$, for $2\leq i\leq h$. Therefore, $b\geq a^{2/h}$, and so, 
\begin{eqnarray*}
\left|N\left(S\right)\right|&\geq&2\alpha\left|N\left(S\right)\cap E(B_1)\right|\geq2\alpha a^{\frac{2}{h}}{n\choose2}\\
&=&a^{\frac{2}{h}}\alpha n(n-1)=a^{\frac{2}{h}}\left(|X|-\beta c\right)\\
&\geq&a^{\frac{2}{h}}\left(|X|-n(n-1)+c\right).
\end{eqnarray*}
The last inequality is a consequence of (\ref{sc}).
\end{proof}
\begin{lemma}\label{g}
If $a^{1-\frac{2}{h}}\leq 1-g$, then $\left|N\left(S\right)\right|\geq\left|S\right|$.
\end{lemma}
\begin{proof}
Since $a^{1-\frac{2}{h}}\leq1-g=\frac{\left|X\right|-n(n-1)+c}{{n\choose h}}$, by Lemma \ref{preg} we have:
\begin{eqnarray*}
\left|N\left(S\right)\right|&\geq& a^{\frac{2}{h}}(\left|X\right|-n(n-1)+c)\\
&\geq& a^{\frac{2}{h}} a^{1-\frac{2}{h}}{n\choose h}
=a{n\choose h}=\left|S\right|.
\end{eqnarray*}
\end{proof}
Let $S'=Y\backslash N(S)$.
\begin{lemma}\label{ns'}
If $|N(S')|\geq|S'|$, then $|N(S)|\geq|S|$.
\end{lemma}
\begin{proof}
Let $|N(S')|\geq|S'|$. Therefore,
$$|N(S)|=|Y|-|S'|\geq|Y|-|N(S')|=|X|-|N(S')|\geq|S|.$$ 
\end{proof}
Now, let $S'_1=S'\cap E(B_1)$.
\begin{lemma}\label{S'}
If $S'\neq\emptyset$, then
$$2\alpha\leq\left|S'\right|\leq\left(2\alpha+2\right)\left|S'_1\right|.$$
\end{lemma}
\begin{proof}
We have 
\begin{eqnarray*}
|S'|&=&\sum_{i=1}^{2\alpha}|S'\cap E(B_i)|+|S'\cap E(H_1)|=2\alpha|S'_1|+|S'\cap E(H_1)|\\
&\leq&2\alpha|S'_1|+2|S'_1|=(2\alpha+2)|S'_1|.
\end{eqnarray*}
Note that the inequality is a result of the fact that every element of $S'$ can be in at most two cycles of $H_1$, since these cycles are chosen from a decomposition. Also, we have $|S'|\geq2\alpha$, since
$$|S'|=\sum_{i=1}^{2\alpha}|S'\cap E(B_i)|+|S'\cap E(H_1)|\geq2\alpha|S'_1|\geq2\alpha.$$
\end{proof}
Now, we are ready to prove our main result of this section.

\noindent {\it \bf Proof of Theorem \ref{ccycledecomposition}.}
If $h\geq\max\left\{c,\left\lceil\frac{n}{2}\right\rceil+1\right\}$,  the result follows from Theorem \ref{largeh}, which settles the case (C). For the remaining cases (when $h< \max\left\{c,\left\lceil\frac{n}{2}\right\rceil+1\right\}$), by Lemma \ref{bipartitegperfectm}, it suffices to prove that $B$ satisfies Hall's condition. Let $\emptyset\neq S\subset X$. In order to show that $|N(S)|\geq|S|$, there are four cases to consider.

{\bf Case A.} $h=3$, $c_n(n-1)$, $n\geq85$, and $L$ is a matching. 

{\bf Case A.1.} $\left|S\right|\leq\left|X\right|-3n(n-1)+c$. 

Note that since $|L|<c$ and $n\geq85\geq4$: 
\begin{eqnarray*}
g&=&\frac{{n\choose 3}-\left|X\right|+n(n-1)-c}{{n\choose 3}}=\frac{|L|+n(n-1)-c}{{n\choose 3}}\\
&\leq&\frac{n(n-1)}{n(n-1)(n-2)/6}=\frac{6}{n-2}\leq3.
\end{eqnarray*}
By Lemma \ref{g}, it is enough to show that $a\leq(1-g)^3$. We have
\begin{eqnarray*}
a{n\choose3}=|S|&\leq&\left|X\right|-3n(n-1)+c=\left|X\right|-2c-3[n(n-1)-c]\\
&\leq&\left|X\right|-2\left[{n\choose 3}-\left|X\right|\right]-3[n(n-1)-c]\\
&=&\left(1-3g\right){n\choose3}\leq\left(1-g\right)^3{n\choose3},
\end{eqnarray*}
where the last inequality is an immediate consequence of $g\leq3$.

{\bf Case A.2.} $\left|S\right|>\left|X\right|-3n(n-1)+c$.

By Lemma \ref{ns'}, it suffices to show that $|N(S')|\geq|S'|$. We have 
\begin{eqnarray}\label{ns'1lessthan3nn-1-c}
\left|N\left(S'_1\right)\right|\leq|X|-|S|<3n(n-1)-c<3n(n-1).
\end{eqnarray}
Let $p,q$ be non-negative integers such that $p<n$, $q<n-\left(p+1\right)$, and $\left|S'_1\right|=pn-{p+1\choose2}+q$. Note that each element of $S'_1$ is a 2-subset of $[n]$ and $N\left(S'_1\right)=\partial_1^+\left(S'_1\right)\backslash L$. Using this and Lemma \ref{KruKatCor} (iii), we have
$$\left|N\left(S'_1\right)\right|\geq\left|\partial_1^+\left(S'_1\right)\right|-|L|\geq p{n-p\choose2}+q(n-p-2)-{q\choose2}-\left|L\right|.$$
 Moreover, since $q<n-p-1$, we have $q(n-p-2)-{q\choose2}\geq0$. Hence, $\left|N\left(S'_1\right)\right|\geq p{n-p\choose2}-\left|L\right|$. 
If $p=9$, then since $L$ is a matching, $\left|N\left(S'_1\right)\right|\geq9{n-9\choose2}-\frac{n}{3}>3n(n-1)$ (for $n\geq85\geq50$), which contradicts (\ref{ns'1lessthan3nn-1-c}). 

If $p>9$, then by Lemma \ref{KruKatCor} (iii), 
\begin{eqnarray*}
\left|N\left(S'_1\right)\right|&\geq&{n-1\choose2}+{n-2\choose2}+\dots+{n-p\choose2}+q(n-p-2)-{q\choose2}-|L|\\
&\geq&{n-1\choose2}+{n-2\choose2}+\dots+{n-9\choose2}-\frac{n}{3}\\
&\geq&9{n-9\choose2}-\frac{n}{3}>3n(n-1),
\end{eqnarray*}
which again contradicts (\ref{ns'1lessthan3nn-1-c}). Therefore, $p\leq8$.

Let $L\left(S'_1\right)=L\cap N(S'_1)$. Since $L$ is a matching by assumption,

$\left|L\left(S'_1\right)\right|\leq \left|S'_1\right|$. Assume that $p=0$. Hence, $\left|S'_1\right|=q$ and $\left|L\left(S'_1\right)\right|\leq q$. Since $n\geq85$, by Lemma \ref{KruKatCor+},
\begin{eqnarray*}
\left|N\left(S'\right)\right|&=&\left|N\left(S'_1\right)\right|\geq\frac{2}{5}qn-\left|L\left(S'_1\right)\right|\geq\frac{2}{5}qn-q\\
&\geq& q\left(\frac{n-2}{3}+2\right)=\left|S'_1\right|\left(\frac{n-2}{3}+2\right)\geq\left|S'\right|.
\end{eqnarray*}
The last inequality uses Lemma \ref{S'} and the fact that $2\alpha\leq\frac{n-2}{3}$ since

\begin{eqnarray*}
\alpha&=&\left\lfloor\frac{{n\choose3}-\left|L\right|}{n(n-1)}\right\rfloor=\left\lfloor\frac{n(n-1)(n-2)/6-\left|L\right|}{n(n-1)}\right\rfloor\\
&\leq&\left\lfloor\frac{n(n-1)(n-2)/6}{n(n-1)}\right\rfloor\leq\frac{n-2}{6}.
\end{eqnarray*}
Now let $p\in[8]$. Again by Lemma \ref{KruKatCor+},
\begin{eqnarray*}
\left|N\left(S'\right)\right|=\left|N\left(S'_1\right)\right|&\geq& p{n-p\choose2}+\frac{2}{5}qn-\left|L\left(S'_1\right)\right|\\
&\geq&\frac{4p}{5}{n\choose2}+\frac{2}{5}qn-\frac{n}{3}.
\end{eqnarray*}
Since 
\begin{eqnarray*}
\frac{4p}{5}{n\choose2}-\frac{n}{3}&\geq&\frac{4p}{5}\frac{n(n-1)}{2}-\frac{2pn}{5}=\frac{12pn(n-2)}{30}\\
&=&\frac{12p}{10}\frac{n(n-2)}{3}\geq\frac{11p}{10}\frac{n(n-2)}{3},
\end{eqnarray*}
and $\frac{2}{5}qn=\frac{12}{30}qn\geq\frac{11}{10}q\frac{n-2}{3}$, we have
$$\left|N\left(S'\right)\right|\geq\frac{11}{10}\frac{n-2}{3}\left(pn+q\right).$$
Moreover, $n\geq85\geq62$ implies $\frac{11}{10}\frac{n-2}{3}\geq\frac{n-2}{3}+2$. Therefore, $\left|N\left(S'\right)\right|\geq\left(\frac{n-2}{3}+2\right)\left|S'_1\right|\geq\left|S'\right|$.

{\bf Case B.} $4\leq h<\max\left\{c,\left\lceil\frac{n}{2}\right\rceil+1\right\}$, $c|n(n-1)$, and $n\geq 23$.

{\bf Case B.1.} $\left|S\right|\leq\left|X\right|-2n(n-1)+c$. 

By Lemma \ref{g} it is enough to show that $a\leq(1-g)^2$. Since $\left|L\right|<c$,
\begin{eqnarray*}
a{n\choose h}=|S|&\leq&\left|X\right|-2n(n-1)+c=\left|X\right|-c-2[n(n-1)-c]\\
&\leq&\left|X\right|-\left[{n\choose h}-\left|X\right|\right]-2[n(n-1)-c]\\
&=&\left(1-2g\right){n\choose h}\leq\left(1-g\right)^2{n\choose h}.
\end{eqnarray*}
 
{\bf Case B.2.} $\left|S\right|>\left|X\right|-2n(n-1)+c$.

{\bf Case B.2.1.} $h=4$.

By Lemma \ref{ns'}, it is enough to show that $|N(S')|\geq|S'|$. We have 
\begin{eqnarray}\label{ns'1lessthan2nn-1-c}
\left|N\left(S'_1\right)\right|\leq|X|-|S|<2n(n-1)-c<2n(n-1).
\end{eqnarray}

Note that each element of $S'_1$ is a 2-subset of $[n]$ and $N(S'_1)=\partial_2^+(S'_1)\backslash L$, or equivalently, $|\partial_2^+(S'_1)|=|N(S'_1)|+|L|$. If $\left|S'_1\right|=7$, then $\left| N\left(S'_1\right)\right|\geq 7{n-8\choose2}+21(n-8)-\left|L\right|$ by Lemma \ref{KruKatCor} (ii). Since $n\geq23$ and $\left|L\right|<c<n$, we have $7{n-8\choose2}+21(n-8)-\left|L\right|\geq\frac{7}{2}(n^2-17n+72)+20n-168\geq2n(n-1)$, which contradicts (\ref{ns'1lessthan2nn-1-c}). Note that if $|S'_1|$ increases, then $|N(S'_1)|$ may not decrease. Therefore, if $|S'_1|>7$, then $|N(S'_1)|\geq2n(n-1)$, which again contradicts (\ref{ns'1lessthan2nn-1-c}). Therefore, $\left|S'_1\right|\leq6$. By Lemma \ref{KruKatCor} (ii), 
\begin{eqnarray*}
\left|N\left(S'_1\right)\right|&\geq&\left|S'_1\right|{n-\left|S'_1\right|-1\choose2}+{\left|S'_1\right|\choose2}\left(n-\left|S'_1\right|-1\right)-\left|L\right|\\
&\geq&\left|S'_1\right|{n-\left|S'_1\right|-1\choose2}-\left|L\right|.
\end{eqnarray*}
Since $\left|S'_1\right|\leq6$, we have ${n-\left|S'_1\right|-1\choose2}\geq{n-7\choose2}$. Therefore, since $\left|L\right|<c<n$, 
$$|N(S')|=\left|N\left(S'_1\right)\right|\geq\left|S'_1\right|{n-7\choose2}-n.$$
By Lemma \ref{S'}, $\left|N\left(S'\right)\right|\geq\frac{{n-7\choose2}}{2\alpha+2}\left|S'\right|-n$.

Since $\alpha=\left\lfloor\frac{{n\choose4}-\left|L\right|}{n(n-1)}\right\rfloor$, 
\begin{eqnarray*}
\alpha&=&\left\lfloor\frac{n(n-1)(n-2)(n-3)/24-\left|L\right|}{n(n-1)}\right\rfloor=\left\lfloor\frac{(n-2)(n-3)}{24}-\frac{\left|L\right|}{n(n-1)}\right\rfloor\\
&\leq&\frac{(n-2)(n-3)}{24}-\frac{\left|L\right|}{n(n-1)}\leq\frac{(n-2)(n-3)}{24}.
\end{eqnarray*}
Hence, $2\alpha\leq\frac{(n-2)(n-3)}{12}$. Therefore,
\begin{eqnarray*}
\left|N\left(S'\right)\right|&\geq&\frac{(n-7)(n-8)}{2}\cdot\frac{12}{(n-2)(n-3)+24}\left|S'\right|-n\\
&=&\frac{6(n-7)(n-8)}{(n-2)(n-3)+24}\left|S'\right|-n\geq2|S'|-n,
\end{eqnarray*}
for $n\geq23>12$. Moreover,
\begin{eqnarray*}
\alpha&=&\left\lfloor\frac{(n-2)(n-3)}{24}-\frac{\left|L\right|}{n(n-1)}\right\rfloor\\
&>&\frac{(n-2)(n-3)}{24}-\frac{\left|L\right|}{n(n-1)}-1\\
&\geq&\frac{(n-2)(n-3)}{24}-2.
\end{eqnarray*}
Therefore, $2\alpha\geq\frac{(n-2)(n-3)}{12}-4$. Consequently, since $n\geq23>11$, we have $2\alpha>n$. Hence, by Lemma \ref{S'}, $\left|S'\right|>n$. Therefore, $\left|N\left(S'\right)\right|>\left|S'\right|$. 

{\bf Case B.2.2.} $5\leq h<\max\left\{c,\left\lceil\frac{n}{2}\right\rceil+1\right\}$.

For every $y\in Y$, 
\begin{eqnarray*}
\left|N\left(y\right)\right|&\geq& {n-2\choose h-2}-\left|L\right|\\
&\geq&{n-2\choose3}\frac{n-5}{h-2}\frac{n-6}{h-3}\cdots\frac{n-h+1}{4}-n\geq{n-2\choose3}-n\\
&\geq&2n(n-1)\geq2n(n-1)-c,
\end{eqnarray*}
since $\left|L\right|<c<n$. The third and fourth inequalities use the facts that $h\leq n-3$, and $n\geq23>22$, respectively. Note that this result means that $y$ has more than $2n(n-1)-c$ neighbors. Therefore, since $\left|S\right|>\left|X\right|-\left[2n(n-1)-c\right]$, every $y\in Y$ has a neighbor in $S$. Hence, $N\left(S\right)=Y$.

{\bf Case D.} $h=n-2$, $c\in\{n-1,n\}$, and $n\geq16$.

Note that 
$$\alpha=\left\lfloor\frac{{n\choose n-2}-|L|}{n(n-1)}\right\rfloor=\left\lfloor\frac{{n\choose2}-|L|}{n(n-1)}\right\rfloor=\left\lfloor\frac{1}{2}-\frac{|L|}{n(n-1)}\right\rfloor=0,$$
where the last equality implies from $|L|<n-1$.

{\bf Case D.1.} $c=n-1$.

Since $\alpha=0$,
$$\beta=\frac{{n\choose2}-|L|}{n-1}=\frac{n(n-1)-2|L|}{2(n-1)}.$$
If $n$ is even, since $|L|<c=n-1$ and $\beta\in\mathbb{N}$, we have $|L|=0$ and so, $\beta=n/2$ and $|Y|={n\choose2}$. If $n$ is odd, then $n-1$ divides $2|L|$, since $\beta\in\mathbb{N}$. Therefore, $|L|=(n-1)/2$, since $n$ is odd and $|L|<n-1$. So, $\beta=(n-1)/2$ and $|Y|={n\choose2}-(n-1)/2$.

Therefore, for any $n$, $|Y|\geq{n\choose2}-\frac{n-1}{2}$. Hence, for any $x\in X$,
\begin{eqnarray*}|N(x)|&\geq&{h\choose2}-\frac{n-1}{2}={n-2\choose2}-\frac{n-1}{2}\\
&=&\frac{n^2-6n+7}{2}\geq\frac{n^2-n}{3}=\frac{2}{3}{n\choose2}\geq\frac{2}{3}|Y|,
\end{eqnarray*}
where the second inequality holds since $n\geq16$.

Also for any $y\in Y$, 
$$|N(y)|\geq{n-2\choose h-2}-|L|\geq{n-2\choose 2}-\frac{n-1}{2}\geq\frac{2}{3}|X|,$$
since $n\geq16$.

Therefore, Hall's condition is satisfied.

{\bf Case D.2.} $c=n$.

Since $\alpha=0$, 
$$\beta=\frac{{n\choose2}-|L|}{n}=\frac{n(n-1)-2|L|}{2n}.$$
If $n$ is odd, since $|L|<c=n$ and $\beta\in\mathbb{N}$, we have $|L|=0$ and so, $\beta=(n-1)/2$ and $|Y|={n\choose2}$. If $n$ is even, then $n$ divides $2|L|$, since $\beta\in\mathbb{N}$. Therefore, $|L|=n/2$, since $n$ is even and $|L|<n$. So, $\beta=n/2$ and $|Y|={n\choose2}-n/2$.

Therefore, for any $n$, $|Y|\geq{n\choose2}-\frac{n}{2}$. Hence, for any $x\in X$,
\begin{eqnarray*}|N(x)|&\geq&{h\choose2}-\frac{n}{2}={n-2\choose2}-\frac{n}{2}\\
&=&\frac{n^2-6n+6}{2}\geq\frac{n(n-1)}{3}=\frac{2}{3}{n\choose2}\geq\frac{2}{3}|Y|,
\end{eqnarray*}
where the second inequality holds since $n\geq16$.

Also for any $y\in Y$, 
$$|N(y)|\geq{n-2\choose h-2}-|L|\geq{n-2\choose 2}-\frac{n}{2}\geq\frac{2}{3}|X|,$$
since $n\geq16$.

Therefore, Hall's condition is satisfied.

{\bf Case E.} $h=n-1$ and $c=n$.

There is nothing to prove, since $K_n^{n-1}$ is itself an $n$-cycle. 

\qed

\section{Final Remarks and Open Problems}
If $\mathcal G$ is an $r$-regular  $h$-uniform $n$-vertex hypergraph with $m$ edges, a simple double counting argument shows that $rn=hm$, so in particular if $\mathcal G$ is $h$-regular, $n=m$. 
Another interesting analogue to Alspach's conjecure on cycle decompositions of $K_n$ in the following problem. 
 \begin{question} \label{prob2regu}
 Find the necessary and sufficient conditions for the existence of a decomposition of $\lambda K_n^h$  into $k$ connected subgraphs $G_1,\dots, G_k$ such that each $G_i$ is  $h$-regular, and  with $c_i$ vertices (or equivalently $c_i$ edges) for $i\in[k]$. 
\end{question}
As the next lemma shows,  not every connected factor is a cycle, and therefore, a solution to Problem \ref{prob2regu}, does not necessarily lead to a solution to Problem \ref{prob1cyc}. 
\begin{lemma} There exists an infinite family of connected regular hypergraphs that are not cycles.
\end{lemma}
\begin{proof}

In \cite{ElMSc}, Ellingham  constructed an infinite family of 3-regular connected bipartite graphs that are not Hamiltonian. Let $G[X,Y]$ an arbitrary graph from this family. Now construct the hypergraph $\mathcal G=(X,Y)$ such that for $x\in X$ and $y\in Y$, $x\in y$ in $\mathcal{G}$ if and only if $x$ and $y$ are adjacent in $G$. It is clear that $\mathcal{G}$ is a connected 3-regular 3-uniform hypergraph. Since $G$ is not Hamiltonian, $\mathcal{G}$ is not a cycle, and the proof is complete.
\end{proof}

Motivated by Corollary \ref{thm1hchypcor2}, we propose the following conjecture.
\begin{conjecture} \label{regHCDconj}
For sufficiently large $n$, $\lambda K_n^h$ can be decomposed into regular Hamiltonian cycles if and only if $n$ divides $\lambda \binom{n}{h}$. 
\end{conjecture}
Note that for $\lambda=1$, this conjecture is more difficult than Bermond's conjecture (which was recently settled by K\"uhn and Osthus \cite{MR3213309}), but it is weaker than the Bailey-Stevens conjecture in which each regular Hamiltonian cycle is required to be $(h-1)$-intersecting \cite{MR2684077}.

\section*{Acknowledgement}
We thank the anonymous referees for their constructive comments. The first author is  grateful to the organizers of EXCILL III: Extremal Combinatorics at Illinois for providing a stimulating environment, and  to Daniela K\"uhn and Deryk Osthus for  a  fruitful discussion.  The first author's research is  supported by PFIG at ISU, and NSA Grant H98230-16-1-0304. The second author wishes to thank the Department of Mathematics, Illinois State University, for its hospitality during her postdoctoral fellowship, when this research was conducted.

\end{document}